\documentclass[12pt]{amsart}
\usepackage{fullpage}   
\usepackage{amssymb}    
\usepackage{amsmath}
\usepackage{amsthm}

\theoremstyle{plain}
\newtheorem{theorem}{Theorem}[section]
\newtheorem{lemma}[theorem]{Lemma}
\newtheorem{definition}[theorem]{Definition}
\newtheorem{proposition}[theorem]{Proposition}
\newtheorem{corollary}[theorem]{Corollary}
\newtheorem{problem}[theorem]{Problem}
\newtheorem{example}[theorem]{Example}
\theoremstyle{definition}
\newtheorem*{acknowledgment}{Acknowledgment}
\newtheorem*{remark}{Remark}

\numberwithin{equation}{section}

\newcommand{\Inn}{\mathrm{Inn}}
\newcommand{\Mlt}{\mathrm{Mlt}}
\newcommand{\Aut}{\mathrm{Aut}}
\newcommand{\gammacat}{\mathbf{\Gamma Lp_{o}}}
\newcommand{\bruckcat}{\mathbf{BrLp_{o}}}

\title{A class of loops categorically isomorphic to Bruck loops of odd order}

\author{Mark~Greer}
\email{mark.greer@du.edu}
\address{Department of Mathematics,
2360 S Gaylord St., University of Denver,
Denver, CO 80208 USA}

\begin{document}
\allowdisplaybreaks

\begin{abstract}
We define a new variety of loops, $\Gamma$-loops.  After showing $\Gamma$-loops are power-associative, our main goal is showing a categorical isomorphism between Bruck loops of odd order and $\Gamma$-loops of odd order.  Once this has been established, we can use the well known structure of Bruck loops of odd order to derive the Odd Order, Lagrange and Cauchy Theorems for $\Gamma$-loops of odd order, as well as the nontriviality of the center of finite $\Gamma$-$p$-loops ($p$ odd).  Finally, we answer a question posed by Jedli\v{c}ka, Kinyon and Vojt\v{e}chovsk\'{y} about the existence of Hall $\pi$-subloops and Sylow $p$-subloops in commutative automorphic loops.
\end{abstract}

\subjclass[2010]{20N05}
\keywords{$\Gamma$-loops, Bruck loops, power-associativity}
\maketitle
\allowdisplaybreaks

\section{Introduction}
\label{sec:intro}
A loop $(Q,\cdot)$ consists of a set $Q$ with a binary operation $\cdot : Q\times Q\to Q$ such that $(i)$ for all $a,b\in Q$, the equations $ax = b$ and $ya = b$ have unique solutions $x,y\in Q$, and $(ii)$ there exists $1\in Q$ such that $1x = x1 = x$ for all $x\in Q$.  Standard references for loop theory are \cite{bruck, hala}.

Let $G$ be a uniquely $2$-divisible group, that is, a group in which the map $x\mapsto x^2$ is a bijection. On $G$ we define two new binary operations as follows:
\begin{align}
x\oplus y &= (xy^2 x)^{1/2}\,, \label{eqn:groupBruck} \\
x\circ y &= xy[y,x]^{1/2}\,.   \label{eqn:groupgamma}
\end{align}
Here $a^{1/2}$ denotes the unique $b\in G$ satisfying $b^2 = a$ and $[y,x] = y^{-1}x^{-1}yx$. Then it turns out that both $(G,\oplus)$ and $(G,\circ)$ are loops with neutral element $1$. Both loops are \emph{power-associative}, which informally means that integer powers of elements can be defined unambiguously. Further, powers in $G$, powers in $(G,\oplus)$ and powers in $(G,\circ)$ all coincide.

For $(G,\oplus)$ all of this is well-known with the basic ideas dating back to Bruck \cite{bruck} and Glauberman \cite{glauberman1}. $(G,\oplus)$ is an example of a \emph{Bruck loop}, that is, it satisfies the following identities
\begin{align}
\tag{Bol} (x\oplus (y\oplus x))\oplus z &= x\oplus (y\oplus (x\oplus z)) \label{eqn:bol}\\
\tag{AIP} (x\oplus y)^{-1} &= x^{-1}\oplus y^{-1}  \label{eqn:aip}
\end{align}

It is not immediately obvious that $(G,\circ)$ is a loop. It is well-known in one special case. If $G$ is nilpotent of class at most $2$, then $(G,\circ)$ is an abelian group (and in fact, coincides with $(G,\oplus)$). In this case, the passage from $G$ to $(G,\circ)$ is called the ``Baer trick'' \cite{martin}.

In the general case, $(G,\circ)$ turns out to live in a variety of loops which we will call $\Gamma$-\emph{loops}. We defer the formal definition until {\S}2, but note here that one defining axiom is commutativity. $\Gamma$-loops include as special cases two classes of loops which have appeared in the literature: commutative RIF loops \cite{KV09} and commutative automorphic loops \cite{JKV1,JKV2,JKV3,deBGV}. We will not discuss RIF loops any further in this paper but we will review the notion of commutative automorphic loop in {\S}2.

Jedli\v{c}ka, Kinyon and Vojt\v{e}chovsk\'{y} \cite{JKV1} showed that starting with a uniquely $2$-divisible commutative automorphic loop $(Q,\circ)$, one can define a Bruck loop $(Q,\oplus_{\circ})$ on the same underlying set $Q$ by
\begin{equation}
\label{eqn:gammatoBruck}
x \oplus_{\circ} y = (x^{-1}\backslash_{\circ} (y^2\circ x))^{1/2}\,.
\end{equation}
Here $a\backslash_{\circ} b$ is the unique solution $c$ to $a\circ c = b$. We will extend this result to $\Gamma$-loops (Theorem \ref{gammatobruck}). This gives us a functor from the category of uniquely $2$-divisible $\Gamma$-loops to the category of uniquely $2$-divisible Bruck loops, which restricts to a functor $\mathcal{B} : \gammacat \rightsquigarrow \bruckcat$ from the category $\gammacat$ of $\Gamma$-loops of odd order to the category $\bruckcat$ of Bruck loops of odd order.
One of our main results is the construction of an inverse functor $\mathcal{G} : \bruckcat \rightsquigarrow \gammacat$, that is, $\mathcal{G}\circ\mathcal{B}$ is the identity functor on $\gammacat$ and $\mathcal{B}\circ\mathcal{G}$ is the identity functor on $\bruckcat$.

Finite Bruck loops of odd order are known to have many remarkable properties, all found by Glauberman \cite{glauberman1,glauberman2}. For instance, they satisfy Lagrange's Theorem, the Odd Order Theorem, the Sylow and Hall Existence Theorems and finite Bruck $p$-loops ($p$ odd) are centrally nilpotent. Using the isomorphism of the categories $\gammacat$ and $\bruckcat$, we immediately get the same results for $\Gamma$-loops of odd order. We work out the details in {\S}6.

Originally, our motivation was to answer an open problem of Jedli\v{c}ka, Kinyon and \allowbreak{Vojt\v{e}chovsk\'{y}} \cite{JKV1}, dealing with the existence of Sylow and Hall subgroups in finite commutative automorphic loops.  The authors showed that a solution would follow from an answer in the odd order case \cite{JKV1}.  Using this and the new isomorphism, the Sylow and Hall Theorems for $\Gamma$-loops of odd order are answered in the affirmative, in a more general way than was originally posed.  Further, the proofs of the Odd Order Theorem and the non triviality of the center of finite $\Gamma$-$p$-loops ($p$ odd) are much simpler than the proofs in \cite{JKV1} and \cite{JKV3} for commutative automorphic loops.

We conclude this introduction with an outline of the rest of the paper. In {\S}2 we give the complete definition of $\Gamma$-loop and we prove that for a uniquely $2$-divisible group $G$, the construction \eqref{eqn:groupgamma} defines a $\Gamma$-loop on $G$. We also give examples of groups $G$ such that $(G,\circ)$ is \emph{not} automorphic.  In {\S}3, we prove that $\Gamma$-loops are power-associative (Theorem \ref{powerassociative}). As a consequence, for $G$ a uniquely $2$-divisible group, powers in $G$ coincide with powers in $(G,\circ)$ (Corollary \ref{powers}). In {\S}4 we review the notion of twisted subgroup of a group and the connection between uniquely $2$-divisible twisted subgroups and Bruck loops of odd order.  In the special case where $(G,\circ)$ is a $\Gamma$-loop constructed on a uniquely $2$-divisible group $G$, it turns out that $(G,\oplus) = (G,\oplus_{\circ})$ (Theorem \ref{samebruck}).  As a consequence, if $(G,\circ)$ is the $\Gamma$-loop of a uniquely $2$-divisible group $G$ and if $(H,\circ)$ is a subloop of $(G,\circ)$, then $H$ is a twisted subgroup of $G$ (Corollary \ref{twistgamma}).

In {\S}5 we construct the functor $\mathcal{G} : \bruckcat\rightsquigarrow\gammacat$ and show that $\mathcal{B}$ and $\mathcal{G}$ are inverses of each other (Theorem \ref{core}). A loop is both a Bruck loop and a $\Gamma$-loop if and only if it is a commutative Moufang loop (Proposition \ref{moufang}) and we observe that restricted to such loops, both $\mathcal{B}$ and $\mathcal{G}$ are identity functors (Proposition \ref{samefrommoufang}).

In {\S}6, using the categorical isomorphism between $\Gamma$-loops of odd order and Bruck loops of odd order, we derive the Odd Order, Sylow and Hall Theorems (Theorems \ref{oddorder}, \ref{sylow}, and \ref{hall}) for $\Gamma$-loops of odd order, as well as the nontriviality of the center of finite $\Gamma$-$p$-loops ($p$ odd).  Finally in {\S}7, we conclude with some open problems.

\section{$\Gamma$-loops}
\label{sec:gamma}
To avoid excessive parentheses, we use the following convention:
\begin{itemize}
\item multiplication $\cdot$ will be less binding than divisions $\backslash, /$.
\item divisions are less binding than juxtaposition
\end{itemize}
For example $xy/z \cdot y\backslash xy$ reads as $((xy)/z)(y\backslash (xy))$.  To avoid confusion when both $\cdot$ and  $\circ$ are in a calculation, we denote divisions by $\backslash_{\cdot}$ and $\backslash_{\circ}$ respectively.

In a loop \emph{Q}, the left and right translations by $x \in Q$ are defined by $yL_{x} = xy$ and $yR_{x} = yx$ respectively.  We thus have $\backslash,/$ as $x\backslash y=yL_{x}^{-1}$ and $y/x=yR_{x}^{-1}$.  We define the \emph{left multiplication group} of $Q$, $\Mlt_{\lambda}(Q)= \left\langle L_{x}\mid x\in Q\right\rangle$ and \emph{multiplication group} of $Q$, $\Mlt(Q)=\left\langle R_{x},L_{x}\mid x\in Q\right\rangle$.  We define the \emph{inner mapping group} of $Q$, $\Inn(Q)=\Mlt(Q)_{1}= \{\theta\in \Mlt(Q) \mid 1\theta=1\}$.  A loop \emph{Q} is an \emph{automorphic loop} if every inner mapping of $Q$ is an automorphism of $Q$, $\Inn(Q) \leq \Aut(Q)$.

In a loop $Q$, we set $x^n = 1 L_x^n$ for all $x\in Q$ and for all $n\in \mathbb{Z}$.  A loop $Q$ is \emph{power-associative} if every $1$-generated subloop is a group.
This is easily seen to be equivalent to $x^m x^n = x^{m+n}$ for every $x\in Q$ and for all $m,n\in \mathbb{Z}$. As noted in the introduction, we informally think of
power-associativity as saying that powers of elements are unambiguously defined.  Bruck loops are power-associative \cite{glauberman1}, and we will show in the next
section that $\Gamma$-loops, defined below, are also power-associative. In the meantime, a special case of power-associativity is the identity $x^{-1} x = 1$,
that is, every element has two-sided inverses. For $\Gamma$-loops, this is immediate from the first part of their definition, which is commutativity.

For finite loops, we can characterize unique $2$-divisibility in different ways.

\begin{theorem}[\cite{JKV1}]
A finite commutative loop $Q$ is uniquely $2$-divisible \emph{if and only if} it has odd order.  Similarly, a finite power-associative loop $Q$ is uniquely $2$-divisible \emph{if and only if} each element of $Q$ has odd order.
\end{theorem}
We now define a new variety of loops, $\Gamma$-loops, which we focus on in this paper.
\begin{definition}
A loop $(Q,\cdot)$ is a $\Gamma$-loop if the following hold
\begin{itemize}
\item[($\Gamma_{1}$)]\quad $Q$ is commutative.
\item[($\Gamma_{2}$)]\quad $Q$ has the automorphic inverse property (AIP):\quad $\forall x,y\in Q$, $(xy)^{-1}=x^{-1}y^{-1}$.
\item[($\Gamma_{3}$)]\quad $\forall x\in Q$, $L_{x}L_{x^{-1}}=L_{x^{-1}}L_{x}$.
\item[($\Gamma_{4}$)]\quad $\forall x,y\in Q$, $P_{x}P_{y}P_{x}=P_{yP_{x}}$ where $P_{x}=R_{x}L_{x^{-1}}^{-1}=L_{x}L_{x^{-1}}^{-1}$.
\end{itemize}
\end{definition}
Note that a loop satisfying the AIP necessarily satisfies $(x\backslash y)^{-1} = x^{-1}\backslash y^{-1}$ and $(x/y)^{-1}=x^{-1}/y^{-1}$. We will use this without comment in what follows.

Our conventions for conjugation and commutators in groups are
\[
x^y = y^{-1}xy \qquad\text{and}\qquad [x,y] = x^{-1}y^{-1}xy = x^{-1} x^y = (y^{-1})^x y\,.
\]
The following identities are easily verified and will be used without reference.

\begin{lemma}
Let $G$ be a group.  Then for all $x,y\in G$,
\begin{itemize}
\item [(i)]\quad $[x,y]^{-1}=[y,x]$
\item [(ii)]\quad $[x,y^{-1}]=[y,x]^{y^{-1}}$ and $[x^{-1},y]=[y,x]^{x^{-1}}$,
\item [(iii)]\quad $[xy,x^{-1}]=[x,yx^{-1}]$,
\item [(iv)]\quad $[x^{-1},y^{-1}]=[x,y]^{(xy)^{-1}}$,
\end{itemize}
Moreover if $G$ is uniquely $2$-divisible,
\begin{itemize}
\item [(v)]\quad $(x^{1/2})^{-1}=(x^{-1})^{1/2}$,
\item [(vi)]\quad $(x^{y})^{1/2}=(x^{1/2})^{y}$.
\end{itemize}
\end{lemma}
\begin{lemma}
Let $G$ be a uniquely $2$-divisible group.  Then
\begin{itemize}
\item[(i)] $x\circ y=y\circ x$,
\item[(ii)] $(x\circ y)^{-1}=x^{-1}\circ y^{-1}$,
\item[(iii)]$xyx= \{x(y\circ x)x(y\circ x)^{-1}\}^{1/2}(y\circ x)$.
\end{itemize} 
\label{twisttrick}
\end{lemma}
\begin{proof}
For (i), we have
\[
x\circ y = xy[y,x]^{1/2} = yx[x,y][y,x]^{1/2}=yx[x,y]([x,y]^{-1})^{1/2}  = yx[x,y]^{1/2} = y\circ x\,.
\]
Similarly for (ii),
\begin{alignat*}{2}
x^{-1}\circ y^{-1} 
&= x^{-1}y^{-1}[y^{-1},x^{-1}]^{1/2} 
&&=(yx)^{-1}([y,x]^{(yx)^{-1}})^{1/2} \\
&=(yx)^{-1}([y,x]^{1/2})^{(yx)^{-1}}
&&=(yx)^{-1} (yx) [y,x]^{1/2} (yx)^{-1}\\
&=[y,x]^{1/2}(yx)^{-1}
&&= ([x,y]^{1/2})^{-1} (yx)^{-1}\\
&=(yx[x,y]^{1/2})^{-1} 
&&=(y\circ x)^{-1}\\
&=(x\circ y)^{-1}.
\end{alignat*}
For (iii), using (i) and (ii) from above,
\begin{alignat*}{2}
yx(y\circ x)^{-1}&=yx(x^{-1}\circ y^{-1}) &&=yxx^{-1}y^{-1}[y^{-1},x^{-1}]^{1/2} \\
&=[y^{-1},x^{-1}]^{1/2} &&=(xyy^{-1}x^{-1}yxy^{-1}x^{-1})^{1/2} \\
&=(xy[y,x](xy)^{-1})^{1/2}
&&=xy[y,x]^{1/2}(xy)^{-1} \\
&=(x\circ y)(xy)^{-1}
&&=(y\circ x)y^{-1}x^{-1} 
\end{alignat*}
Hence we have
\begin{alignat*}{2}
\{xyx(y\circ x)^{-1}\}^2 &= x\underbrace{yx(y\circ x)^{-1}}xyx(y\circ x)^{-1}
&&= x(y\circ x)y^{-1}x^{-1}xyx(y\circ x)^{-1} \\
&= x(y\circ x)x(y\circ x)^{-1}\,. &&
\end{alignat*}
Thus $xyx = \{x(y\circ x)x(y\circ x)^{-1}\}^{1/2}(y\circ x)$, as claimed.
\end{proof}

\begin{theorem}
Let $G$ be a uniquely $2$-divisible group.  Then $(G,\circ)$ is a $\Gamma$-loop.
\label{newloop}
\end{theorem}
\begin{proof}
To see $(Q,\circ)$ is a loop, fix $a,b\in Q$ and let $x=\{a^{-1}ba^{-1}b^{-1}\}^{1/2}b$.  Thus, we compute
\begin{align*}
x&=\{a^{-1}ba^{-1}b^{-1}\}^{1/2}b &\Leftrightarrow\\
(xb^{-1})^2&=a^{-1}ba^{-1}b^{-1} &\Leftrightarrow\\
xb^{-1}x&=a^{-1}ba^{-1} &\Leftrightarrow\\
xa&=bx^{-1}a^{-1}b &\Leftrightarrow\\
[x,a]&=(x^{-1}a^{-1}b)^{2} &\Leftrightarrow\\
ax[x,a]^{1/2}&=b &\Leftrightarrow\\
a\circ x&=b.
\end{align*}
Note that this gives the following expression for $\backslash_{\circ}$:
\[
a\backslash_{\circ} b = \{a^{-1}ba^{-1}b^{-1}\}^{1/2} b\,.
\]
\noindent 
It is easy to see that inverses coincide in $G$ and $(G,\circ)$.  Therefore, $(\Gamma_{1})$ and $(\Gamma_{2})$ are exactly Lemma \ref{twisttrick}(i) and (ii).  For $(\Gamma_{3})$, first note
\begin{equation}
x^{-1}\circ (xy)=y[xy,x^{-1}]^{1/2}=y[x,yx^{-1}]^{1/2}=(yx^{-1})\circ x=x\circ(yx^{-1}) .
\label{A}
\end{equation}
Similarly,
\begin{equation}
x^{-1}\circ y=x^{-1}y[y,x^{-1}]^{1/2}=x^{-1}y([x,y]^{1/2})^{x^{-1}}=y[y,x][x,y]^{1/2}x^{-1}=y[y,x]^{1/2}x^{-1}.
\label{B}
\end{equation}
Therefore
\begin{equation*}
x^{-1}\circ (x\circ y)=x^{-1}\circ(xy[y,x]^{1/2})\stackrel{\eqref{A}}{=}x\circ ((y[y,x]^{1/2})x^{-1})\stackrel{\eqref{B}}{=}x\circ(x^{-1}\circ y).
\end{equation*}
For $(\Gamma_{4})$, rewriting Lemma \ref{twisttrick}(iii) gives 
\[xyx=\{x(y\circ x)x(y\circ x)^{-1}\}^{1/2}(y\circ x)=x^{-1}\backslash_{\circ}(y\circ x)=yP_{x}. \]
Let $y\Psi_{x}=xyx$, that is, $y\Psi{x}=yP_{x}$.  Hence, $P_{x}P_{y}P_{x}=\Psi_{x}\Psi_{y}\Psi_{x}=\Psi_{y\Psi_{x}}=P_{yP_{x}}$.
\end{proof}
\begin{lemma}
Commutative automorphic loops are $\Gamma$-loops.
\end{lemma}
\begin{proof}
This follows from Lemmas $2.6$, $2.7$ and $3.3$ in \cite{JKV1}.
 \end{proof}
\begin{example}
The smallest known $\Gamma$-loop constructed from a group of odd order has order $375$, and its underlying group is the smallest group of odd order that is not metabelian, with GAP library number $[375; 2]$.  Later we will show an example of a subloop of order $75$ which is also not automorphic, and that subloop is the smallest known nonautomorphic $\Gamma$-loop of odd order.
\label{e1}
\end{example}
\begin{example}
The following is the smallest $\Gamma$-loop which is neither a commutative automorphic nor commutative RIF loop, found by \textsc{Mace4} \cite{PM}.
\begin{table}[h]
\centering
\begin{tabular}{r|rrrrrr}
$\cdot$ & 0 & 1 & 2 & 3 & 4 & 5\\
\hline
    0 & 0 & 1 & 2 & 3 & 4 & 5 \\
    1 & 1 & 0 & 3 & 5 & 2 & 4 \\
    2 & 2 & 3 & 0 & 4 & 5 & 1 \\
    3 & 3 & 5 & 4 & 0 & 1 & 2 \\
    4 & 4 & 2 & 5 & 1 & 0 & 3 \\
    5 & 5 & 4 & 1 & 2 & 3 & 0
\end{tabular}
\end{table}
\label{smallest}
\end{example}

\section{$\Gamma$-Loops are power-associative}
\label{sec:PA}
Recall our definition $x^{n}=1L_{x}^{n}$ for all $n\in \mathbb{Z}$.
\begin{proposition}
\label{prp:invworks}
Let $Q$ be a $\Gamma$-loop.  Then $x^{-n}=(x^{-1})^n=(x^{n})^{-1}$.
\end{proposition}
\begin{proof}
The first equality, $(1)L_x^{-n} = (1) L_{x^{-1}}^n$, is equivalent to $ 1= (1) L_{x^{-1}}^n L_x^n$.  By $(\Gamma_{3})$, $L_{x^{-1}}^n L_x^n = ( L_{x^{-1}} L_x )^n$.  But since $L_{x^{-1}} L_x \in \Inn(Q)$, we are done.  The second equality follows from $(\Gamma_{2})$.
 \end{proof}
\begin{proposition}
Let $Q$ be a $\Gamma$-loop.  Then
\begin{equation}\tag{$P_{1}$}
P_{x}=L_{x}L_{x^{-1}}^{-1}=L_{x^{-1}}^{-1}L_{x}
\label{p1}
\end{equation}
\begin{equation}\tag{$P_{2}$}
P_{x}L_{x}=L_{x}P_{x}
\label{p2}
\end{equation}
\end{proposition}
\begin{proof}
These follow from $(\Gamma_{3})$.
 \end{proof}
\begin{lemma}
Let $Q$ be a  $\Gamma$-loop.  Then $\forall k,n\in \mathbb{Z}$ we have the following:
\begin{itemize}
\item[(a)]\quad $x^nP_{x}=x^{n+2}$
\item[(b)]\quad $P_{x}^{n}=P_{x^{n}}$
\item[(c)]\quad $x^{k}P_{x^{n}}=x^{k+2n}$
\end{itemize}
\label{paidentity}
\end{lemma}

\begin{proof}
Note that $1P_{x}=x^{2}$ by $(\Gamma_{3})$.  For all $n$, we have
\begin{equation*}
x^{n}P_{x}=1L_{x}^{n}P_{x}\stackrel{\eqref{p2}}{=}1P_{x}L_{x}^{n}=x^{2}L_{x}^{n}=1L_{x}^{2}L_{x}^{n}=1L_{x}^{n+2}=x^{n+2}\,.
\end{equation*}

For (b), the cases $n=0,1$ are trivially true.  For $n>1$  ,
\begin{equation*}
P_{x}^{n}=P_{x}P_{x}^{n-2}P_{x}=P_{x}P_{x^{n-2}}P_{x}\stackrel{(\Gamma_{4})}{=}P_{x^{n-2}P_{x}}\stackrel{(a)}{=}P_{x^{n}}\,.
\end{equation*}
If $n=-1$ then $P_{x^{-1}}=L_{x^{-1}}L_{x}^{-1}=(L_{x}L_{x^{-1}}^{-1})^{-1}=P_{x}^{-1}$.  Thus we have for any $n<0$,
\begin{equation*}
P_x^n = (P_x^{-n})^{-1} = P_{x^{-n}}^{-1} = P_{(x^{-n})^{-1}} = P_{x^n}\,,
\end{equation*}
by Proposition \ref{prp:invworks}.

 For (c), let $k$ be fixed.  Then
\begin{equation*}
x^k P_{x^n}\stackrel{(b)} {=}  x^k P_x^n\stackrel{(a)}{=} x^{k+2} P_x^{n-1}\stackrel{(a)}{=} \ldots\stackrel{(a)} {=} x^{k+2n}\,.\qedhere
\end{equation*}
 \end{proof}

\noindent
For $m\in \mathbb{N}_0 = \mathbb{N}\cup \{0\}$, we define PA($m$) to be the statement:
\begin{equation*}
\forall i \in \{-m,...,m\} \text{ and } \forall j \in \{-m-1,...,m+1\},\quad
x^i x^j = x^{i+j}\,.
\end{equation*}

\begin{lemma}
Let $Q$ be a $\Gamma$-loop.  Then PA($m$) holds for all $m\in\mathbb{N}_0$.
\label{patrick}
\end{lemma}

\begin{proof}
We induct on $m$.  PA($0$) is obvious.  Assume PA($m$) holds for some $m\geq 0$. We establish PA($m+1$) by proving $x^i x^j = x^{i+j}$ for each of the following cases:
\begin{itemize}
\item[(1)]\quad $i \in \{-m-1,\ldots, m+1\}$,\quad $j \in \{-m,\ldots, m\}$,
\item[(2)]\quad $i \in \{-m,\ldots,m\}$,\quad $j = m+1$\quad or\quad $j=-m-1$,
\item[(3)]\quad $i = m+1$, $j = -m-1$\quad or\quad $i=-m-1$, $j = m+1$,
\item[(4)]\quad $i = m+1$, $j = m+1$\quad or\quad $i = -m-1$, $j = -m-1$,
\item[(5)]\quad $i \in \{-m-1,\ldots, m+1\}$,\quad $j=m+2$\quad or\quad $j=-m-2$.
\end{itemize}
By $(\Gamma_2)$ and Proposition \ref{prp:invworks}, $x^i x^j = x^{i+j}$ implies $x^{-i} x^{-j} = x^{-i-j}$. So in each of cases (2), (3), (4) and (5), we only need to establish one of the subcases.

Case (1) follows from PA($m$) (with the roles of $i$ and $j$ reversed) and commutativity. Case (2) also follows from PA($m$). Case (3) follows from Proposition \ref{prp:invworks}: $x^{m+1}x^{-m-1}=x^{m+1}x^{-(m+1)}=1$.

For case (4),
\[
x^{m+1} x^{m+1} = (1)L_{x^{-(m+1)}}^{-1} L_{x^{m+1}} \stackrel{\eqref{p1}}{=} (1)P_{x^{m+1}} \stackrel{(\ref{paidentity}c)}{=} x^{2m+2}\,.
\]

Finally, for case (5), first suppose $i\in \{-m-1,\ldots,-1\}$. Then $-2m-2\leq 2i\leq -2$, and so $-m\leq m+2+2i\leq m$, that is, $m+2+2i\in \{-m,\ldots,m\}$. Thus
\[
x^i x^{m+2} = (x^{m+2})P_{x^i} L_{x^{-i}} \stackrel{(\ref{paidentity}c)}{=} x^{-i}x^{m+2+2i} \stackrel{\textnormal{PA($m$)}}{=} x^{m+2+i}\,.
\]
Now suppose $i\in \{1,\ldots,m+1\}$. Then $-2m-2\leq -i\leq -2$, and so $-m\leq m+2-2i\leq m$, that is, $m+i-2i\in \{-m,\ldots,m\}$. Thus
\[
x^i x^{m+2} \stackrel{(\ref{paidentity}c)}{=} (x^{m+2-2i})P_{x^i} L_{x^i} \stackrel{\eqref{p2}}{=} (x^i x^{m+2-2i})P_{x^i}
\stackrel{\textnormal{PA($m$)}}{=} (x^{m+2-i})P_{x^i} \stackrel{(\ref{paidentity}c)}{=} x^{m+2+i}\,. \qedhere
\]
\end{proof}

\begin{theorem}
$\Gamma$-loops are power-associative.
\label{powerassociative}
\end{theorem}
\begin{proof}
This follows immediately from Lemma \ref{patrick}.  Indeed, $x^k x^{\ell}=x^{k+\ell}$ with $0 \leq |k| \leq |\ell |$ follows from PA($|\ell |$).
 \end{proof}

By Theorem \ref{newloop} and Theorem \ref{powerassociative}, for a uniquely $2$-divisible group $G$ and its corresponding $\Gamma$-loop $(G,\circ)$, we have powers coinciding.
\begin{corollary}
Let $G$  be a uniquely $2$-divisible group and $(G,\circ)$ its associated $\Gamma$-loop.  Then powers in $G$ coincide with powers in $(G,\circ)$.
\label{powers}
\end{corollary}

\section{Twisted subgroups and uniquely $2$-divisible Bruck loops}
\label{sec:twist}
We turn to an idea from group theory, first studied by Aschbacher \cite{aschbacher}.  We follow the notations and definitions used by Foguel, Kinyon and Phillips \cite{FKP}, and refer the reader to that paper for a more complete discussion of the following results.
\begin{definition}
A \emph{twisted subgroup} of a group $G$ is a subset $T\subset G$ such that
$1\in T$ and for all $x,y\in T$, $x^{-1}\in T$ and $xyx \in T$.
\end{definition}
\begin{example}[\cite{FKP}]
Let $G$ be a group and $\tau \in \Aut(G)$ with  $\tau^{2}=1$.  Let $K(\tau)=\{g\in Q\mid g\tau=g^{-1}\}$.  Then $K(\tau)$ is a twisted subgroup.
\end{example}
\begin{proposition}
Let $G$ be uniquely $2$-divisible group and let $\tau\in \Aut(G)$ satisfy $\tau^{2}=1$.  Then $K(\tau)$ is closed under $\circ$ and $\backslash_{\circ}$ and hence is a subloop of $(G,\circ)$.
\label{twistloop}
\end{proposition}
\begin{proof}
Let $x,y\in K(\tau)$.  Then
\[
(x\circ y)\tau =(xy[y,x]^{1/2})\tau
=x\tau y\tau [y\tau,x\tau]^{1/2}
=x^{-1}y^{-1}[y^{-1},x^{-1}]^{1/2}
=x^{-1}\circ y^{-1} = (x\circ y)^{-1}
\]
by ($\Gamma_2$). Similarly, ($\Gamma_2$) also gives
$(x\backslash_{\circ} y)\tau = (x\backslash_{\circ} y)^{-1}$.
\end{proof}

\begin{theorem}[\cite{FKP}]
Let $Q$ be a Bruck loop.  Then $L_{Q}$ is a twisted subgroup of $\Mlt_{\lambda}(Q)$.  If $Q$ has odd order, then $\Mlt_{\lambda}(Q)$ has odd order.  Moreover, there exists a unique $\tau \in \Aut(\Mlt_{\lambda}(Q))$ where $\tau^{2}=1$ and $L_{Q}=\{\theta\in \Mlt_{\lambda}(Q) \mid \theta \tau =\theta^{-1} \}$.  On generators, $(L_x)^{\tau}=L_{x^{-1}}$.
\label{twistbruckloop}
\end{theorem}
\begin{corollary}
Let $(Q,\cdot)$ be a Bruck loop of odd order.  Then $(L_{Q},\circ)$ is a $\Gamma$-loop.
\label{leftmultloop}
\end{corollary}
\begin{proof}
This follows from Proposition \ref{twistloop} and Theorem \ref{twistbruckloop}.
 \end{proof}

We have a bijection from $Q$ to $L_{Q}$ given by $x\mapsto 1L_{x}$.  This allows us to define a $\Gamma$-loop operation directly on $Q$ as follows:
\begin{equation*}
x\circ y = 1(L_{x}\circ L_{y})
\end{equation*}
where we reuse the same symbol $\circ$.  By construction, the $\Gamma$-loops $(L_{Q},\circ)$ and $(Q,\circ)$ are isomorphic.
\begin{proposition}
Let $(Q,\cdot)$ be a Bruck loop of odd order.  Then $(Q,\circ)$ is a $\Gamma$-loop.  Moreover, powers in $(Q,\circ)$ coincide with powers in $(Q,\cdot)$.
\end{proposition}
\begin{proof}
For powers coinciding, suppose $x^{n}$ denotes powers in $(Q,\cdot)$.  Since Bruck loops are left power-alternative \cite{robinson1}, $x^{n}=1L_{x^n}=1L_{x}^{n}$.  By Corollary \ref{powers}, $L_{x}^{n}$ coincides with the $n$th power of $L_{x}$ in $(L_{Q},\circ)$.  Thus $x^{n}$ is the $n$th power of $x$ in $(Q,\circ)$.  Since this argument is clearly reversible, we have the desired result.
 \end{proof}

For a uniquely $2$-divisible $\Gamma$-loop $(Q,\cdot)$, we set
\[
x \oplus_{\cdot} y = (x^{-1}\backslash_{\cdot} (y^2 x))^{1/2}\,,
\]
and if $\circ$ is another $\Gamma$-loop operation on the same underlying set, we similarly define $\oplus_{\circ}$. Our next goal is to generalize Lemma $3.5$ of \cite{JKV1} and show that $(Q,\oplus_{\cdot})$ is a Bruck loop.

\begin{lemma}
Let $Q$ be a $\Gamma$-loop.  Then
\begin{equation*}
(yP_{x})^{2}=x^{2}P_{y}P_{x}.
\end{equation*}
\label{gammatrick}
\end{lemma}
\begin{proof}
By Proposition \ref{paidentity}$(a)$, we have that $x^{2}=1P_{x}$.  Hence, $x^{2}P_{y}P_{x}=1P_{x}P_{y}P_{x}\stackrel{(\Gamma_{4})}{=}1P_{yP_{x}}=(yP_{x})^{2}$ by Proposition \ref{paidentity}$(a)$ again.
 \end{proof}
\begin{theorem}
Let $(Q,\cdot)$ be a uniquely $2$-divisible $\Gamma$-loop.  Then $(Q,\oplus_{\cdot})$ is a Bruck loop.  Moreover, powers in $(Q,\cdot)$ coincide with powers in $(Q,\oplus_{\cdot})$.
\label{gammatobruck}
\end{theorem}
\begin{proof}Note that $(x\oplus_{\cdot} (y\oplus_{\cdot} x))\oplus_{\cdot} z = x\oplus_{\cdot} (y\oplus_{\cdot} (x\oplus_{\cdot} z))$ is equivalent to $\lambda_{x}\lambda_{y}\lambda_{x}=\lambda_{x\oplus_{\cdot}(y\oplus_{\cdot} x)}$ where $y\lambda_{x}=x\oplus_{\cdot} y$.  Let $x\delta=x^{2}$.  Then $y\lambda_{x}=x\oplus_{\cdot} y = (x^{-1}\backslash_{\cdot}(y^{2}x))^{1/2} = y\delta P_{x}\delta^{-1}$.  Thus,
\begin{equation*}
\lambda_{x}\lambda_{y}\lambda_{x}=\delta P_{x}\delta^{-1}\delta P_{y}\delta^{-1}\delta P_{x}\delta^{-1}=\delta P_{x}P_{y}P_{x}\delta^{-1}\stackrel{(\Gamma_{4})}{=}\delta P_{yP_{x}} \delta^{-1}.
\end{equation*}
But by Proposition \ref{gammatrick},
\begin{equation*}
yP_{x}=(x^{2}P_{y}P_{x})^{1/2}=(x^{-1}\backslash_{\cdot}[(y^{-1}\backslash_{\cdot}(x^{2}y))x])^{1/2}=x\oplus_{\cdot}(y^{-1}\backslash_{\cdot}(x^{2}y))^{1/2}=x\oplus_{\cdot}(y\oplus_{\cdot} x).
\end{equation*}
Thus,
\begin{equation*}
\lambda_{x}\lambda_{y}\lambda_{x}=\delta P_{yP_{x}} \delta^{-1}=\delta P_{x\oplus_{\cdot} (y\oplus_{\cdot} x)}\delta^{-1}=\lambda_{x\oplus_{\cdot} (y\oplus_{\cdot} x)}.
\end{equation*}
The fact that $(Q,\oplus_{\cdot})$ has AIP is straightforward from $(\Gamma_{2})$.  Powers coinciding follows from power-associativity of $(Q,\cdot)$ and $(Q,\oplus_{\cdot})$.
 \end{proof}

We now have a construction of $\Gamma$-loops from Bruck loops and a construction of Bruck loops from $\Gamma$-loops. In the next section,
we will show that when we iterate these constructions, we get nothing new, but in the meantime, we will use the following notation conventions.
Our ``starting loop'' will always be denoted by $(Q,\cdot)$. The Bruck loops constructed from a particular $\Gamma$-loop operation will be
distinguished by subscripts. The $\Gamma$-loop operation constructed from any Bruck loop will be denoted simply by $\circ$; as it turns out, we
will not need to construct $\Gamma$-loops for (seemingly) distinct Bruck loops.
 
So for instance, if we start with a Bruck loop, construct a $\Gamma$-loop and then another Bruck loop, we will follow this sequence:
\[
(Q,\cdot)\rightsquigarrow (Q,\circ)\rightsquigarrow (Q,\oplus_{\circ})\,.
\]
If we start with a $\Gamma$-loop, construct a Bruck loop and then a $\Gamma$-loop, we will follow this sequence:
\[
(Q,\cdot)\rightsquigarrow (Q,\oplus_{\cdot})\rightsquigarrow (Q,\circ)
\]
 
All of this is just a temporary inconvenience, as our goal in the next section is to show that the starting and ending loops in both sequences
are not only isomorphic, they are in fact identical.

Given a Bruck loop $(Q,\cdot)$ of odd order, we wish to give the explicit equation of the left division operation in $(Q,\circ)$.  We will need the following two facts for Bol loops, both well known.
\begin{proposition}[\cite{glauberman1, robinson1}]
In a Bruck loop $Q$, the identity $(xy)^{2}=x\cdot y^{2}x$ holds for all $x,y\in Q$.
\label{bolp1}
\end{proposition}
\begin{proposition}
Let $(Q,\cdot)$ be a Bruck loop of odd order and let $(Q,\circ)$ be its $\Gamma$-loop.  For all $a,b\in Q$,
\[
b/_{\circ}a = (a^{-1}b^{1/2})/_{\cdot}b^{-1/2}\,.
\]
\label{division}
\end{proposition}
\begin{proof}
Let $a,b\in Q$ be fixed and set $x=(a^{-1}b^{1/2})/_{\cdot}b^{-1/2}$.  Then $xb^{-1/2}=a^{-1}b^{1/2}$.  Squaring both sides gives 
\[x\cdot b^{-1}x = a^{-1} \cdot ba^{-1}\] 
using Proposition \ref{bolp1}.  But this is equivalent to $L_{x\cdot b^{-1}x}=L_{a^{-1}\cdot ba^{-1}}$ and since $(Q,\cdot)$ is a Bruck loop, we have $L_{x}L_{b}^{-1}L_{x}=L_{a}^{-1}L_{b}L_{a}^{-1}$.  This in turn is equivalent to $[L_{a},L_{x}]=(L_{a}^{-1}L_{x}^{-1}L_{b})^{2}$ and therefore $L_{x}L_{a}[L_{a},L_{x}]^{1/2}=L_{b}$.  That is, $L_{x}\circ L_{a}=L_{b}$.  Hence, $1(L_{x}\circ L_{a})=1L_{b}$ and so $x\circ a = b$.
\end{proof}

Let $(G,\cdot)$ be a uniquely $2$-divisible group.  We have its Bruck loop $(G,\oplus)$ and also the Bruck loop $(G,\oplus_{\circ})$ of the $\Gamma$-loop $(G,\circ)$.  We now show these coincide.
\begin{theorem}
Let $(G,\cdot)$ be a uniquely $2$-divisible group.  Then $(G,\oplus)=(G,\oplus_{\circ})$.
\label{samebruck}
\end{theorem}
\begin{proof}
Recall by Lemma \ref{twisttrick}(iii), we have $xyx=yP_{x}$ for all $x,y\in G$.  Replacing $y$ by $y^{2}$ and applying square roots gives $x\oplus y=(xy^{2}x)^{1/2}=(y^{2}P_{x})^{1/2}=(x^{-1}\backslash_{\circ}(y^{2}\circ x))^{1/2}=x\oplus_{\circ}y$.
 \end{proof}

\begin{corollary}
Let $(G,\cdot)$ be a uniquely $2$-divisible group, let $(H,\circ)\leq (G,\circ)$ and suppose that $H$ is closed under taking square roots.  Then $H$ is a twisted subgroup of $G$.  In particular, if $G$ is a finite group of odd order and $(H,\circ)\leq (G,\circ)$, then $H$ is a twisted subgroup of $G$.
\label{twistgamma}
\end{corollary}
\begin{proof}
Again we have $xyx=yP_{x}\in H$ for all $x,y\in H$.  Finally, since powers coincide in $H$ and $(H,\circ)$, $x^{-1}\in H$.
\end{proof}
\begin{example}
Let $G$ be the smallest nonmetabelian group of odd order from Example \ref{e1}.  Then there exists a twisted subgroup $H$ of $G$ with $|H|=75$.  Here $(H,\circ)$ is the smallest known example of a nonautomorphic $\Gamma$-loop of odd order.
\end{example}

\section{Inverse functors}
\label{sec:inverse}
We will need the following lemma for our main result.
\begin{lemma}
Let $(Q,\cdot)$ be a uniquely $2$-divisible $\Gamma$-loop and $(Q,\oplus_{\cdot})$ be its Bruck loop.  Then
\begin{equation}
x\oplus_{\cdot}(xy)^{-1/2}=y^{-1}\oplus_{\cdot} (xy)^{1/2}\,.
\end{equation}
\label{unique1}
\end{lemma}

\begin{proof}
First note that $x\oplus_{\cdot}(xy)^{-1/2}=y^{-1}\oplus_{\cdot} (xy)^{1/2} \Leftrightarrow x^{-1}\backslash (x^{-1}y^{-1}\cdot x)=y\backslash (xy\cdot y^{-1})$.  Therefore we compute
\begin{alignat*}{3}
x^{-1}\backslash (x^{-1}y^{-1}\cdot x)
&\stackrel{(\Gamma_{1})}{=}x^{-1}\backslash(x\cdot x^{-1}y^{-1})
&&\stackrel{(\Gamma_{3})}{=}x^{-1}\backslash(x^{-1}\cdot xy^{-1})
&&=xy^{-1}\\
&\stackrel{(\Gamma_{1})}{=}y^{-1}x=y\backslash(y\cdot y^{-1}x)
&&\stackrel{(\Gamma_{3})}{=}y\backslash(y^{-1}\cdot yx)
&&\stackrel{(\Gamma_{1})}{=}y\backslash(yx\cdot y^{-1})\,.\qedhere
\end{alignat*}
\end{proof}

Now let $\mathcal{G} : \bruckcat \rightsquigarrow \gammacat$ be the functor given on objects by assigning to each Bruck loop of odd order $(Q,\cdot)$ its corresponding $\Gamma$-loop $(Q,\circ)$, and let $\mathcal{B} : \gammacat \rightsquigarrow \bruckcat$ be the functor given on objects by assigning to each $\Gamma$-loop of odd order $(Q,\cdot)$ its corresponding Bruck loop $(Q,\oplus_{\cdot})$.

\begin{theorem}{\ }
\begin{itemize}
\item [(A)] $\mathcal{G}\circ\mathcal{B}$ is the identity functor on $\gammacat$.
\item [(B)] $\mathcal{B}\circ\mathcal{G}$ is the identity functor on $\bruckcat$.
\end{itemize}
\label{core}
\end{theorem}
\begin{proof}
(A) Let $(Q,\cdot)$ be a $\Gamma$-loop of odd order, let $(Q,\oplus_{\cdot})$ be its corresponding Bruck loop and let $(Q,\circ)$ be the $\Gamma$-loop of $(Q,\oplus_{\cdot})$.
Lemma \ref{unique1} and Proposition \ref{division} imply
\[
x = (x\oplus_{\cdot} (xy)^{-1/2}) /_{\oplus_{\cdot}} (xy)^{-1/2}
= (y^{-1}\oplus_{\cdot} (xy)^{1/2}) /_{\oplus_{\cdot}} (xy)^{-1/2} = (xy)/_{\circ} y.
\]
Thus $xy = x\circ y$, as claimed.

(B) Let $(Q,\cdot)$ be a Bruck loop of odd order, let $(Q,\circ)$ be its corresponding $\Gamma$-loop and let $(Q,\oplus_{\circ})$ be the Bruck loop of $(Q,\circ)$.
Recalling that the map $x\mapsto L_x$ (left translations in $(Q,\cdot)$) is an isomorphism of $(Q,\circ)$ with $(L_Q,\circ)$, we have
\begin{alignat*}{3}
L_{(x\oplus_{\circ} y)^2} &= L_{x^{-1}\backslash_{\circ} (y^2\circ x)}
&&= L_x^{-1} \backslash_{\circ} (L_y^2 \circ L_x)
&&= (L_x \oplus_{\circ} L_y )^2 \\
&= (L_x \oplus L_y )^2
&&= L_x L_y^2 L_x
&&= L_{x\cdot(y^2\cdot x)} \\
&= L_{(xy)^2}\,, && &&
\end{alignat*}
using Theorem \ref{samebruck} and Proposition \ref{bolp1}. Thus
$(xy)^2 = (x\oplus_{\circ} y)^2$ and so the desired result follows from taking square roots.
\end{proof}

We note in passing that we have proven a result which can be stated purely in terms of Bruck loops of odd order:

\emph{Let $(Q,\cdot)$ be a Bruck loop of odd order.  For each $x,y\in Q$, the equation
\[
xz^{-1/2}=y^{-1}z^{1/2}
\label{uniq}
\]
has a unique solution $z\in Q$.}  Indeed, $z = x\circ y$ where $(Q,\circ)$ is the $\Gamma$-loop of $(Q,\cdot)$.

We conclude this section by discussing the intersection of the varieties of Bruck loops and $\Gamma$-loops.
\begin{proposition}
A loop is both a Bruck loop and $\Gamma$-loop if and only if it is a commutative Moufang loop.
\label{moufang}
\end{proposition}
\begin{proof}
The ``if" direction is clear.  For the converse, commutative Bruck loops are commutative Moufang loops \cite{robinson2}.
\end{proof}
The following result quickly follows from the fact that Moufang loops are diassociative (\emph{i.e.} the subloop $\langle x,y \rangle$ is a group for all $x,y$) and the definitions of the operations.
\begin{proposition}
Let $(Q,\cdot)$ be a uniquely $2$-divisible commutative Moufang loop.  Then $(Q,\cdot)=(Q,\circ)=(Q,\oplus_{\cdot})$.
\label{samefrommoufang}
\end{proposition}

\begin{proposition}
Let $(Q,\cdot)$ be a $\Gamma$-loop of exponent $3$.  Then $(Q,\cdot)$ is a commutative Moufang loop.
\end{proposition}
\begin{proof}
The associated Bruck loop $(Q,\oplus_{\cdot})$ is a commutative Moufang loop \cite{robinson2}.  Moreover, recalling Proposition \ref{unique1}, $x\oplus_{\cdot}(xy)^{-1/2}=y^{-1}\oplus_{\cdot}(xy)^{1/2}$ holds for all $x,y\in Q$.  Hence, using diassociativity, we have
\[
x=(y^{-1}\oplus_{\cdot}(xy)^{1/2})\oplus_{\cdot}(xy)^{1/2}=y^{-1}\oplus_{\cdot}(xy).
\] 
Thus, $xy=y\oplus_{\cdot} x=x\oplus_{\cdot} y$, and therefore, $(Q,\cdot)=(Q,\oplus_{\cdot})$ is a commutative Moufang loop.
 \end{proof}

\section{$\Gamma$-loops of odd order}
\label{sec:structure}
In this section we will take notational advantage of Theorem \ref{core} and write simply $\oplus$ for the Bruck loop operation of a $\Gamma$-loop of odd order.

\begin{proposition}
Let $(Q,\cdot)$ be a $\Gamma$-loop with $|Q|=p^{2}$ for $p$ prime.  Then $(Q,\cdot)$ is an abelian group.
\end{proposition}
\begin{proof}
Loops of order $4$ are abelian groups \cite{hala}, so assume $p>2$.  For odd primes, Bruck loops of order $p^{2}$ are abelian groups \cite{burn}.  Thus since $(Q,\oplus)$ is an abelian group, so is its $\Gamma$-loop, which, by Theorem \ref{core}, coincides with $(Q,\cdot)$.
 \end{proof}

\begin{lemma}
Let $(Q,\cdot)$ be a $\Gamma$-loop of odd order and let $(Q,\oplus)$ be its Bruck loop.  Then the derived subloops of $(Q,\cdot)$ and $(Q,\oplus)$ coincide. In particular, the derived series of $(Q,\cdot)$ and $(Q,\oplus)$ coincide.
\label{derive}
\end{lemma}
\begin{proof}
By the categorical isomorphism (Theorem \ref{core}), any normal subloop of $(Q,\oplus)$ is a normal subloop of $(Q,\cdot)$ and vice versa.  If $S$ is the derived subloop of $(Q,\oplus)$, then $S$ is a normal subloop of $(Q,\cdot)$ such that $(Q/S,\cdot)$ is an abelian group.  If $M$ were a smaller normal subloop of $(Q,\cdot)$ with this property, then it would have the same property for $(Q,\oplus)$, a contradiction.  The converse is proven similarly.
 \end{proof}

\begin{theorem}[Odd Order Theorem]
$\Gamma$-loops of odd order are solvable
\label{oddorder}
\end{theorem}
\begin{proof}
Let $(Q,\cdot)$ be a $\Gamma$-loop of odd order and let $(Q,\oplus)$ be its Bruck loop.
Then $(Q,\oplus)$ is solvable (\cite{glauberman2}, Theorem 14(b), p. 412), and so the desired result follows from Lemma \ref{derive}.
\end{proof}

\begin{theorem}[Lagrange and Cauchy Theorems]
Let $(Q,\cdot)$ be a $\Gamma$-loop of odd order.  Then:
\begin{itemize}
\item [(L)] If $A\leq B\leq Q$ then $|A|$ divides $|B|$.
\item [(C)] If an odd prime $p$ divides $|Q|$, then $Q$ has an element an order of $p$.
\end{itemize}
\label{lag_cauchy}
\end{theorem}
\begin{proof}
Both subloops $A$ and $B$ give subloops $(A,\oplus)$ and $(B,\oplus)$ of $(Q,\oplus)$.  The result follows from (\cite{glauberman1}, Corollary 4, p. 395).  Similarly, if an odd prime $p$ divides $|Q|$, then $(Q,\oplus)$ has an element of order $p$  (\cite{glauberman1}, Corollary 1, p. 394).  Hence, $Q$ has an element of order $p$.
\end{proof}

\begin{theorem}
Let $Q$ be a $\Gamma$-loop of odd order and let $p$ be an odd prime.  Then $|Q|$ is a power of $p$ if and only if every element of $Q$ has order a power of $p$.
\label{powerofp}
\end{theorem}
\begin{remark}
Note that this is false for $p=2$ by Example \ref{smallest}.
\end{remark}
\begin{proof}
If $|Q|$ is a power of $p$, then by Theorem \ref{lag_cauchy}(L) every element has order a power of $p$.  On the other hand, if $|Q|$ is divisible by an odd prime $q$, then by Theorem \ref{lag_cauchy}(C), $Q$ contains an element of order $q$.  Therefore, if every element is order $p$, $|Q|$ must have order a power of $p$.
 \end{proof}

Thus, in the odd order case, we can define $p$-subloops of $\Gamma$-loops.  Moreover, we can now show the existence of Hall $\pi$-subloops and Sylow $p$-subloops.

\begin{theorem}[Sylow subloops]
$\Gamma$-loops of odd order have Sylow $p$-subloops.
\label{sylow}
\end{theorem}
\begin{proof}
Let $(Q,\cdot)$ be a $\Gamma$-loop of odd order and $(Q,\oplus)$ its Bruck loop.  Then $(Q,\oplus)$ has a Sylow $p$-subloop (\cite{glauberman1}, Corollary 3, p. 394), say $(P,\oplus)$.  But then $(P,\circ)$ is a Sylow $p$-subloop of $(Q,\cdot)$ by Theorem \ref{core}.
\end{proof}

\begin{theorem}[Hall subloops]
$\Gamma$-loops of odd order have Hall $\pi$-subloops.
\label{hall}
\end{theorem}
\begin{proof}
Let $(Q,\cdot)$ be a $\Gamma$-loop of odd order and $(Q,\oplus)$ its Bruck loop.  Then $(Q,\oplus)$ has a Hall $\pi$- subloop (\cite{glauberman1}, Theorem 8, p. 392), say $(H,\oplus)$.  But then $(H,\circ)$ is a Hall $\pi$-subloop of $(Q,\cdot)$ by Theorem \ref{core}.
\end{proof}

Recall the \emph{center} of a loop $Q$ is defined as
\begin{equation*}
Z(Q)=\{a\in Q\mid xa=ax, \quad ax\cdot y=a\cdot xy, \quad xa\cdot y=x\cdot ay \quad\text{and}\quad  xy\cdot a=x\cdot ya \quad \forall x,y\in Q\}.
\end{equation*}
\begin{theorem}
Let $(Q,\cdot)$ be a Bruck loop of odd order.  Then $Z(Q,\cdot)=Z(Q,\circ)$.
\end{theorem}
\begin{proof}
Let $a\in Z(Q,\cdot)$ and recall $a(a\circ x)^{-1/2}=x^{-1}(a\circ x)^{1/2}$ from Lemma \ref{unique1} holds for any $x\in Q$.  Then
\begin{equation*}
x\cdot a(a\circ x)^{-1/2}=(a\circ x)^{1/2}\Leftrightarrow xa\cdot (a\circ x)^{-1/2}=(a\circ x)^{1/2} \Leftrightarrow xa=a\circ x.
\end{equation*}
Moreover, for any $x,y,z\in Q$, 
\begin{equation*}
z[L_{y},L_{xa}]=zL_{y}^{-1}L_{xa}^{-1}L_{y}L_{xa}=xa \cdot y((xa)^{-1}\cdot y^{-1}z)=x\cdot y( x^{-1}\cdot y^{-1}z)=z[L_{y},L_{x}].
\end{equation*}
Thus, for all $x,y \in Q$, noting $L_{ax}=L_{a}L_{x}$,
\begin{alignat*}{3}
(a\circ x)\circ y
&=ax\circ y
&&=L_{ax}\circ L_{y}
&&=L_{a}L_{x}\circ L_{y} \\
&=L_{a}L_{x}L_{y}[L_{y},L_{a}L_{x}]^{1/2}
&&=L_{a}L_{x}L_{y}[L_{y},L_{x}]^{1/2}
&&=L_{a}L_{x\circ y} \\
&=L_{a(x\circ y)}
&&=L_{a\circ (x\circ y)}
&&=a\circ (x\circ y)\,.
\end{alignat*}
Therefore $a\in Z(Q,\circ)$ by commutativity of $(Q,\circ)$.
Similarly, let $a\in Z(Q,\circ)$ and let $(Q,\oplus)$ be its corresponding Bruck loop. It is enough to show that $ax=xa$ and $xa\cdot y=x\cdot ay$ since in a Bruck loop, $xa\cdot y=x\cdot ay \Leftrightarrow a\cdot xy=ax\cdot y$.  We compute
\begin{equation*}
ay=a\oplus y=(a^{-1}\backslash_{\circ}(y^{2}\circ a))^{1/2} =(a^{2}\circ y^{2})^{1/2}=a\circ y=y\circ a=ya.
\end{equation*}
Moreover,
\begin{align*}
xa\cdot y&=xa \oplus y=((xa)^{-1}\backslash_{\circ}(y^{2}\circ (xa)))^{1/2}= ((x\circ a)^{-1}\backslash_{\circ}(y^{2}\circ (x\circ a)))^{1/2}\\
&=(x^{-1}\backslash_{\circ}((a\circ y)^{2} \circ x))^{1/2}=x\oplus (ay)=x\cdot ay.
\end{align*}
Therefore $a\in Z(Q,\cdot)$.
 \end{proof}

Define $Z_{0}(Q)=1$ and $Z_{n+1}(Q), n\geq 0$ as the preimage of $Z(Q/Z_{n}(Q))$ under the natural projection.  This defines the \emph{upper central series}
\begin{equation*}
1\leq Z_{1}(Q)\leq Z_{2}(Q)\leq\ldots\leq Z_{n}(Q)\leq \ldots \leq Q
\end{equation*}
of $Q$.  If for some $n$ we have $Z_{n-1}(Q)<Z_{n}(Q)=Q$, then $Q$ is said to be \emph{(centrally) nilpotent of class n}.
\begin{theorem}
Let $p$ be an odd prime.  Then finite $\Gamma$ $p$-loops are centrally nilpotent.
\end{theorem}
\begin{proof}
Since $Z(Q,\cdot) = Z(Q,\oplus)$, it follows by induction that $Z_{n}(Q,\cdot)=Z_{n}(Q,\oplus)$ for all $n>0$.  But $(Q,\oplus)$ is centrally nilpotent of class, say, $n$ (\cite{glauberman1}, Theorem 7, p. 390).  Therefore, $(Q,\cdot)$ is centrally nilpotent of class $n$.
\end{proof}

\section{Conclusion}
\label{sec:end}
It is natural to ask what conditions on a uniquely $2$-divisible group $G$ would force $(G,\circ)$ to be a commutative automorphic loop.  When $G$ is metabelian of odd order, $(G,\circ)$ is conjectured to be a commutative automorphic loop.  Our only examples of a group $G$ with $(G,\circ)$ not a commutative automorphic loop occur when $G$ is nonmetabelian.
\begin{problem}
Let $G$ be a uniquely $2$-divisible metabelian group.  Is $(G,\circ)$ a commutative automorphic loop?
\label{p4}
\end{problem}
\noindent
Moreover, using the above remarks, we can ask:
\begin{problem}
Let $(Q,\circ)$ be a commutative automorphic loop and let $(Q,\oplus)$ be the corresponding Bruck loop.  Is the left multiplication group of $(Q,\oplus)$ metabelian?
\end{problem}

\begin{problem}
Let $(Q,\cdot)$ be a Bruck loop of order $p^{3}$ where $p$ is an odd prime.  Is $(Q,\circ)$ is a commutative automorphic loop?
\end{problem}
\noindent
If this holds, then a classification of Bruck loops of order $p^{3}$ would follow from \cite{deBGV} and Theorem \ref{core}.

\begin{acknowledgment}
I would like to thank Michael Kinyon for suggesting this problem and for his careful reading of each draft of the manuscript.  Some investigations in this paper were assisted by the automated deduction tool \textsc{Prover9} \cite{PM}, the finite model builder \textsc{Mace4} \cite{PM} and GAP with the LOOPS package \cite{GAP, GAPNV}.
\end{acknowledgment}

\nocite{JKNV}

\bibliographystyle{amsplain}
\bibliography{bib}
\end{document}